\documentclass[12pt]{article}

\usepackage{amssymb}
\usepackage{amsmath}
\usepackage{graphicx}
\usepackage{enumerate}
\usepackage{mathrsfs}
\usepackage{latexsym}
\usepackage{psfrag}
\usepackage{mathrsfs}
\usepackage{hyperref}

\newcommand{\qed}{\hfill $\square$\\
\vspace{0.1cm}}

\newtheorem{theorem}{Theorem}[section]

\newtheorem{proposition}[theorem]{Proposition}
\newtheorem{corollary}[theorem]{Corollary}

\newenvironment{proof}{\noindent{\em Proof.}}{\qed}

\newcommand{\gc}{\gamma_c}
\newcommand{\gcqn}{\gamma_c(Q_n)}

\newcommand{\comments}[1]{}

\begin{document}
\title{Spanning Trees and Domination in Hypercubes}

\author{Jerrold R. Griggs
\thanks{Department of Mathematics, University of South Carolina,
Columbia, SC, USA 29208 (griggs@math.sc.edu).
Research supported in part by a grant from the Simons Foundation (\#282896 to Jerrold Griggs)
. }  }
\date{\today}

\maketitle

\begin{abstract}
Let $L(G)$ denote the maximum number of leaves in any spanning tree of a
connected graph $G$.  We show the (known) result that for the $n$-cube $Q_n$,
$L(Q_n) \sim 2^n = |V(Q_n)|$ as $n\rightarrow \infty$.
Examining this more carefully, consider
the minimum size of a connected dominating set of vertices $\gcqn$, which
is $2^n-L(Q_n)$ for $n\ge2$.
We show that $\gcqn\sim 2^n/n$.
We use Hamming codes and an ``expansion" method
to construct leafy spanning trees in $Q_n$.

\end{abstract}


\section{Introduction}\label{sec:Over}

The $n$-cube graph $Q_n$ has $2^n$ vertices, the strings
$a_1\ldots a_n$ on $n$ bits, where two vertices
are adjacent if and only if their strings differ in exactly
one coordinate (where one vertex has 0 and the other has 1).
The $n$-cube is frequently used as a structure for computer
networks, where there are $2^n$ processors corresponding
to the vertices of $Q_n$.  An efficient way to connect all
of the processors, so that they all communicate with each
other, is to take a spanning tree in $Q_n$.

With this in mind, S. Bezrukov imagined it would be
interesting to construct such spanning trees with many
leaves (degree one vertices).
At the IWOCA conference (Duluth, 2014), Bezrukov
proposed the following problem:
Letting $L(G)$
denote the maximum number of leaves in any spanning tree
of a connected simple graph $G$, what can one say about
$L(Q_n)$?
He shared this problem in notes~\cite{Bez}.

For a spanning tree, the non-leaf vertices are connected,
so form a tree themselves, which we may think of as
the backbone of the tree:  All vertices are either in
this backbone, or are leaves adjacent to it.
Bezrukov's question then is equivalent to constructing
a spanning tree of the hypercube with the smallest
backbone.

Notice that the opposite question,
finding the {\em minimum\/} number of
leaves in a spanning tree, is easy: By a simple induction
$Q_n$ has a Hamilton path for all $n\ge1$.
This path is a spanning tree with just two leaves.
We are interested in the other extreme, {\it maximizing\/} the
number of leaves.

Our problem is closely related to the subject of
domination in graphs.
A subset $W$ of the vertex set $V$ of a graph $G=(V,E)$
is a {\em dominating set\/} if
every vertex is either in $W$ or adjacent to
some vertex in $W$.
The {\em domination number\/} $\gamma(G)$ is the
minimum size of any dominating set.

Note that if one pulls off the leaves from a spanning
tree $T$ for a connected graph $G=(V,E)$ with at least
three vertices, then the
remaining vertices $W$ form a dominating set, and,
moreover, what remains of $T$ still connects them.
That is, $W$ forms a connected dominating set.
Conversely, from any connected dominating set we can
span them with a tree and attach any other vertices as
leaves to obtain a spanning tree.
The minimum size of a connected dominating set of $G$ is
called the {\em connected domination number\/} $\gamma_c(G)$.

We see that maximizing the number of leaves of
any spanning tree of such $G$ corresponds
to minimizing the size of a connected dominating set.
From this discussion we obtain for such $G$
\[
L(G) + \gc(G) = |V(G)|.
\]
The simple ordering relationship between these parameters is
\[
1\le \gamma(G)\le \gamma_c(G)\le |V|.
\]

For example, one can readily check that
for the four-cycle $Q_2$,
$\gamma=\gamma_c=L=2$, while for the ordinary
cube~$Q_3$,
$\gamma=2, \gamma_c=L=4$.
For larger $n$ more than half the vertices can be leaves.

The earliest paper we can find that investigates the connected
domination number of a graph is by Sampathkumar and Walikar (1979)
~\cite{SamWal}.
Several studies investigate bounding $L(G)$ for classes of graphs
$G$,  such as those with given
minimum degree \cite{Sto,GriKleSha,KleWes,GriWu}.
Caro {\em et al.\/}~\cite{CarWesYus}\/study both parameters,
and provide more references.
Many papers concern algorithms for finding leafy trees
(or small connected dominating sets).

Searching online we discovered several papers concerning
domination in hypercubes.
These were often done independently of other studies.
The 1990 dissertation of Jha~\cite{Jha} gives a good
general upper bound on $\gamma(Q_n)$,
which is just twice the easy lower bound.
Arumugam and Kala~\cite{AruKal} (1998) focus on domination in
hypercubes.
Duckworth {\em et al.}~\cite{DDGZ}(2001) give good general bounds
on $L(Q_n)$.  It follows that $L(Q_n)\sim 2^n$.
It means asymptotically there is a spanning tree for the
hypercube in which almost all vertices are leaves.
It is nicer to restate their results in terms of connected
domination:

\begin{theorem}\label{thm:previous}~\cite{DDGZ}

\begin{itemize}
\item  Lower bound:  For $n\ge2$, $\frac{\gcqn}{2^n} \ge \frac{1}{n}$
\item  Upper bound:  As $n\rightarrow\infty$, $\frac{\gcqn}{2^n} \le (1 + o(1)) \frac{2}{n}$
\end{itemize}
\end{theorem}

Another 2012 study of hypercubes~\cite{CheSyu} gives values of
$\gcqn$ for small $n$, but unfortunately its formula for
general $n$, stated without proof, is far from correct.
Mane and Waphare~\cite{ManWap} investigate several
generalizations of domination numbers of hypercubes.
The 2017 Master's thesis of Kubo\v n~\cite{Kub}
considers domination in hypergraphs, and uses some of the
same methods as in this paper.

In the next section, we present simple general lower bounds on
$\gamma(Q_n)$ and $\gcqn$.
In Section~\ref{sec:Hamming} we describe the Hamming code
construction that gives a ``perfect dominating set" for
$Q_n$ when $n$ is of the form $2^k-1$.
We give a method to produce a small connected dominating
set, given a dominating set, that leads to an upper bound
on $\gcqn$ for $n=2^k-1$.
A simple inductive method we call doubling is used to
give upper bounds on $\gamma(Q_n)$ and $\gcqn$ for
general~$n$ in Section~\ref{sec:Doubling}.

Where we make new progress is by introducing
in Section~\ref{sec:Exp} a new method
we call expansion, in which we take a minimum dominating
set in each of $2^j$ copies of $Q_N$ and connect them
appropriately to obtain a small connected dominating set in
$Q_n$, where $n=N+j$.
Choosing $N$ and $j$ wisely improves
the best previous upper bound on
$\gcqn$ by a factor of 2.  Indeed, in Section~\ref{sec:Main}
we settle the leading asymptotic behavior of $\gcqn$:

\begin{theorem}\label{thm:main}
As $n\rightarrow\infty$,
$\frac{\gcqn}{2^n} = (1+o(1))\frac{1}{n}.$
\end{theorem}

Restating this in terms of the maximum number of
leaves, it means

\[
L(Q_n)= (1-\frac{1}{n}+o(\frac{1}{n}))2^n.
\]

We conclude with suggestions for further study and
acknowledgements of valuable ideas and support of this
project.

\section{Domination Lower Bounds}\label{sec:Lower}

Let us note some easy lower bounds on our domination parameters for the hypercube $Q_n$.

\begin{proposition}\label{thm:lower}

\begin{itemize}
\item For $n\ge1$, $\gamma(Q_n)\ge 2^n/(n+1).$
\item For $n\ge2$, $\gcqn\ge (2^n-2)/(n-1)\ge 2^n/n.$

\end{itemize}
\end{proposition}

\begin{proof}
A single vertex can dominate at most itself and its $n$ neighbors, leading to the
lower bound on $\gamma(Q_n)$.

Next, consider a connected dominating set of $Q_n$ of size $c$.  There is a tree $T$
on these $c$ vertices using $c-1$ edges from $Q_n$.  The sum of degrees of these $c$
vertices has $2c-2$ accounted for by $T$.  It means that the number of additional
vertices (dominated by those in $T$) is at most $nc-2(c-1)$.  But there are
$2^n-c$ vertices besides $T$.  Rearranging terms gives the stated inequality on
$c$, hence the lower bounds on $\gcqn$.
\end{proof}

\section{Hamming Code}\label{sec:Hamming}

The famous Hamming code gives an elegant construction of a ``perfect dominating set"
in $Q_n$ when $n=2^k-1$ for some integer $k\ge1$.
This means it achieves the lower bound on $\gamma(Q_n)$ in Proposition~\ref{thm:lower}.
Viewing the vertices of $Q_n$ for such an $n$ as $n$-dimensional vectors over
$GF(2)$, the code consists of the $2^{n-k}$ vectors in the row space of a
$(n-k)\times n$ matrix built as follows:
The first $n-k$ columns form the identity matrix, while the rows of the other
$k$ columns consist of all $n-k=2^k-k-1$
vectors of length $k$ with weight (number of ones) at least 2.
The difference between any two vectors in this row space is then a
nonzero vector in the row space, and hence a nonempty sum of rows of the
matrix.  By design, such a sum will always have weight at least three.

Consequently, the $2^{n-k}$ stars in $Q_n$ that are centered at the vectors
in the row space are disjoint.  Each star is a $K_{1,n}$.  By counting,
we see that these stars partition the vertices of $Q_n$.
They form a minimum dominating set for $Q_n$.

As Bezrukov pointed out when he proposed his problem about $L(Q_n)$, for such $n$
we only have to add some edges between leaves of different stars to obtain a
spanning tree with many leaves.  After all, $Q_n$ is connected, and all edges not
used in the stars are between leaves of stars (different stars, in fact).
If we have $c$ components, we need to add $c-1$ edges to obtain a spanning tree;
here, $c=2^{n-k}$.
At worst, each additional edge costs us two new leaves--it would be less, if we are
able to use several edges from the same leaf.
When we finish, we have a spanning tree where the non-leaves are the $c$ star centers
from the Hamming code, as well as at most $2c-2$ vertices that were star leaves.

In fact, we can use this method for any connected simple graph $G$ to build a
spanning tree.  Starting from a minimum dominating set of $c$ vertices, the stars
centered at those vertices cover the entire vertex set (though in general they
are not disjoint, and dominating vertices could even be adjacent).  One can add
at most $c-1$ edges between stars to create a spanning tree.  We obtain this
general bound:

\begin{proposition}\label{thm:ggc}
Let $G$ be a connected simple graph.  Then
\[
\gamma_c(G)\le 3\gamma(G)-2.
\]
\end{proposition}

Applying this to our Hamming code construction, we obtain

\begin{proposition}\label{thm:hamming}
Let $n=2^k-1$, where the integer $k\ge1$.  Then
$\gamma(Q_n)=2^{n-k}=2^n/(n+1)$, and $\gcqn< (3/(n+1))2^n.$
\end{proposition}

For this Hamming code case $n=2^k-1$  our tree construction can be
viewed this way:  Starting from a perfect dominating set in $Q_n$, we
take the corresponding $C=2^n/(n+1)$ stars $K_{1,n}$ and add
$C-1$ edges to form a tree with many leaves.
Since all edges for the star centers are used already, each
edge we add will join leaves from two different stars.  At worst, we give up
$2(C-1)$ star leaves (they become part of the tree backbone), plus
the backbone contains the $C$ star centers.
This gives us a connected dominating set of size at most $3C-2\sim 3(2^n/n)$.

If we are fortunate, we don't have to pick two new leaves for each
successive additional edge:  It could be that one or both leaves are already
in the backbone.  However, for each of the $C$ stars we must give up at least
one leaf, in order that the stars connect in the spanning tree.  It means
that the connected dominating set we construct will have at least
$2C\sim 2(2^n/n)$ vertices.

\section{Doubling}\label{sec:Doubling}

So far, we have constructed leafy trees in the $n$-cube only when $n$ has the
special form $2^k-1$.  The $(n+1)$-cube can be viewed as built from two copies
of $Q_n$, with a matching of edges joining the corresponding vertices from each
copy.  This is true for any value of $n$, not just the special values where the
Hamming code exists.

If we take a dominating set for each copy of $Q_n$, we get a dominating
set for $Q_{n+1}$.  Moreover, if we take the same connected dominating set
for each copy, it gives a dominating set for $Q_{n+1}$ that is connected.
We see this simply by adding the edge joining the two copies of a vertex in
the connected dominating set for $Q_n$.  We record these observations about
doubling:

\begin{proposition}\label{thm:doubling}
For all $n\ge1$, $\gamma(Q_{n+1})\le 2\gamma(Q_n)$, and
$\gamma_c(Q_{n+1})\le 2\gcqn$.
\end{proposition}

Now suppose $n$ is between two consecutive values where the Hamming code
construction is the last section applies, say $n=N+j$, where $k\ge1$,
$N=2^k-1$, and $0\le j\le 2^k$.
We apply the doubling proposition $j$ times, starting from
$Q_N$, and obtain:
\[
\gamma(Q_n)\le 2^j \frac{2^N}{N+1}=\frac{N+j}{N+1} \frac{2^n}{n}< 2\frac{2^n}{n}.
\]
It follows that
\[
\gamma(Q_n)/2^n<2/n\rightarrow 0,
\]
as $n\rightarrow\infty$.
This matches the bound given by Jha~\cite{Jha}.

For connected domination we apply Proposition~\ref{thm:ggc} and obtain:
\[
\gcqn<3\gamma(Q_n)<6\frac{2^n}{n}.
\]
It follows that
\[
\frac{\gcqn}{2^n} < \frac{6}{n} \rightarrow0,
\]
as $n\rightarrow\infty$, confirming our earlier
assertion that there are spanning trees
for hypercubes with almost all vertices being leaves.
Of course, Theorem~\ref{thm:previous} got a better bound than this on
$\gcqn/2^n$;  Our
main result will do even better.

Let us summarize our findings so far.  The domination problem for
$Q_n$ is solved by the Hamming code for $n=N=2^k-1$.
Then as $n=N+j$ grows with $j, 0\le j\le 2^k$, our upper bound
on $\gamma(Q_n)/(2^n/n)$ grows from around 1 to around 2.
However, at $j=2^k$, we have the next Hamming code case,
$n=2^{k+1}-1$, and it is better to switch again to the Hamming
code construction.
It means we have a sawtooth function upper bound, rising from 1 to 2
as $n$ increases, then abruptly dropping back down to 1 and rising
again.  Of course, each tooth covers an interval of length about
$2^k$, so the teeth get wider with $k$.

Owing to our upper bound Proposition~\ref{thm:ggc}, for
connected domination $\gcqn$ has a similar sawtooth upper
bound, but each tooth rises from value 3 to 6.

\section{Expansion}\label{sec:Exp}

We introduce a new method of tree construction that takes advantage
of small dominating sets to produce smaller connected dominating
sets in $Q_n$.
This will  bring down our upper bound for
connected domination, and eventually allow us to solve our
problem asymptotically.

For constructing a spanning tree,
the Hamming code bound punished us by potentially using up
so many leaves to connect the stars.
If we repeatedly double the construction, then it repeats
this penalty over and over.  A better idea could then be
to select one copy (or ``layer") of the base hypercube,
add edges to connect the stellar clusters in just that
layer, and then connect all the copies of each star
center to the one in the special layer.

Describing this explicitly,
let $N=2^k-1$, and  $n=N +j$,
where $0\le j\le 2^k$.
Partition the vertices of $Q_n$ into $2^j$ ''layers"
according to the last $j$ coordinates of the vertices
$(a_1,\ldots,a_n)$.
Each layer induces a $Q_N$, and its vertices are partitioned
into $|C|=2^{N-k}$ stars, according to the Hamming code
partition of $Q_{N}$.
For each star $S$ in the partition of $Q_N$, there are
$2^j$ vertices, one in each layer, that are centers of
the stars corresponding to star $S$.  The centers all
agree in their first $N$ coordinates, so together induce
a subgraph $Q_j$.
By adding $2^j-1$ edges these stars (copies of $S$)
can be connected into a tree.
We now have a forest of $|C|=2^{N-k}$ such trees.

We connect these trees by adding $|C|-1$ edges,
which may as well all be
in the layer ending with 0's.
Each such edge adds at most two vertices
to the connected dominating set we construct.
It is similar to how we
connected the stars in the Hamming code construction.
We record the result of our expansion construction:

\begin{proposition}\label{thm:Exp}
Let $n=N+j$, where $N=2^k-1$ and $1\le j\le 2^k$.
Then $\gamma_c(Q_n)\le 2^j|C|+2(|C|-1),$
where $C$ is the set of $2^{N-k}$ codewords for
the Hamming code in $Q_{N}$.
\end{proposition}

We have seen that $\gcqn/2^n\ge1/n$ for all $n\ge2$.
It would
be nice if we could find a tree construction for $Q_n$ that
has so many leaves that its backbone (connected dominating set)
comes close to achieving the lower bound, acting asymptotically
like a perfect dominating set:  What we want is that
$\gcqn/(2^n/n)\rightarrow 1$ as  $n\rightarrow\infty$.
Expansion allows us to come much closer to this goal.
Here is what we can show now:

\begin{theorem}\label{thm:ExpAsy}
For all $n\ge1$, $\gamma_c(Q_n)/2^n<2/n$.
For all $n\ge3$, $\gamma_c(Q_n)/2^n>1/n.$
We have
$\liminf_{n\rightarrow\infty} \gcqn/(2^n/n)=1$.
\end{theorem}

\begin{proof}
We have $n, N, K, j$ as above.  Proposition~\ref{thm:Exp} gives us

\begin{align*}
\gamma_c(Q_n)
&\le 2^j|C|+2(|C|-1)\\
&< (2^j+2)|C|\\
&=(2^j+2)(2^{N-k})\\
&=(2^n + 2^{N+1})/2^k .\\
\end{align*}

We rewrite this as
\[
\frac{\gcqn}{2^n/n}<\left(1+\frac{1}{2^{j-1}}\right)
\left(1+\frac{j-1}{2^k}\right).
\]

In our range $1\le j\le 2^k$, the first term in the product
on the right starts at 2 and decreases exponentially quickly
towards 1.
The second term starts at 1 and grows linearly to just below 2
at the end of this range.
Throughout this whole range in $j$, the product is at most
$2$, giving us the first statement of the theorem.

The second statement, the lower bound on $\gcqn/2^n$, follows
easily from Proposition~\ref{thm:lower}.
For the third statement, we select values of $n$ for which we
can show $\gcqn/(w^n/n)$.  Specifically, given $k$ take
$j=k+1$, so that $n=2^k+k$.  Then in the upper bound inequality
above on $\gcqn/(2^n/n)$, both terms in the product are small
(slightly above 1), and their product $\rightarrow 1$ as
$k\rightarrow\infty$.
The $\liminf$ statement follows.
\end{proof}

An interesting observation is that for $n$ of the
form $2^k-1$, the Hamming code exists, but the
corresponding spanning tree construction for $Q_n$
we described earlier only guarantees that
$\gamma_c(n)/(2^n/n)$ is at most
$3$ for such $n$.  We can do better, constructing a
tree that reduces the bound for such $n$ to 2,
by taking the Hamming
construction for $2^{k-1}-1$ and applying the expansion
method with $j=2^{k-1}$.
Nevertheless, we are still seeking to do better,
aiming to construct trees that bring the bound
down to 1 asymptotically.

\section{Main result}\label{sec:Main}

We have shown how to construct spanning trees for
hypercubes $Q_n$ with many leaves--the proportion of
the $2^n$ vertices that are not leaves is at most
roughly $2/n$.  The idea is to take a Hamming code and
then expand.

Now observe that the expansion idea can be used starting
from {\em any\/} values of $N$, not just a Hamming code value
$2^k-1$, and from {\em any\/} dominating set $C$ in $Q_N$, to
produce a connected dominating set for $Q_n$, $n=N+j$:
Set $C$ gives a partition of $Q_N$ into
stars.  For each star center (vertex in the dominating
set), add edges to connect the $2^j$ copies of the vertex.
In the original $Q_N$ add edges to connect the stars.
We now have a spanning tree for $Q_n$.
Denote by $D$ its backbone, a  connected
dominating set in $Q_n$.
We get an upper bound on $|D|$ as in
Proposition~\ref{thm:Exp}.
Assuming $|C|$ is minimum-sized, we get that
\[
\gcqn< (2^j+2)\gamma(Q_N).
\]

Given $n$ large, let $j$ be an integer close to
$\log n$ (logarithm base 2), and take $N=n-j$.
Then the display above implies that
$\gcqn/(2^n/n)$ is bounded above approximately by
$\gamma(Q_n)/(2^N/N)$.
So an upper bound function for the domination number,
shifted to the right by $\log n$,
yields an approximate upper bound function on
connected domination.

In particular, if it holds that for domination
$\gamma(Q_n)/(2^n/n)$ tends towards 1, its
lower limit, then the same will be true for the similar
expression for connected domination!
Fortunately, what we need is proven
in the 1997
book {\em Covering Codes\/} by Cohen, Honkala, Litsyn, and
Lobstein~\cite{CHLL}, p.332.  They attribute the result to Kabatyanskii
and Panchenko~\cite{KabPan} (1988).  The proof relies on
various coding constructions, including $q$-ary Hamming codes
for prime powers $q$.  It also depends on results on the
density of primes.

We include their result on the domination number as the first
part of our Main Theorem below.  It is restated for convenient
comparison
to our result for connected domination number, the second part,
which can be viewed as a strengthening of the first part.

\begin{theorem}\label{thm:Main}
The domination ratio for hypercubes satisfies~\cite{KabPan}

\[
{\lim_{n\rightarrow\infty}
 \frac{\gamma(Q_n)}{2^n/n}=1.}
 \]

The connected domination ratio satisfies
\[
{\lim_{n\rightarrow\infty}
 \frac{\gamma_c(Q_n)}{2^n/n}=1.}
 \]
\end{theorem}

\begin{proof}

As noted above, the first statement is proven in the literature.
What is new is the second part, which is a stronger statement.
Building on Theorem~\label{thm:ExpAsy} it suffices to give an upper bound on
$\gcqn/(2^n/n)$ that goes to 1 as $n\rightarrow\infty$.
As in the discussion above, given $n$
we take $j$ is close to $\log n$ and $N=n-j$.
Given $\varepsilon>0$ we have that for all sufficiently large $n$
(and $N$) that
\[
\frac{\gamma(Q_N)}{2^N/N}<1+\varepsilon.
\]
Applying this in the discussion above, gives us for all
sufficiently large $n$ that
\[
\frac{\gcqn}{2^n/n}<(1+\varepsilon)^2,
\]
and the second part follows.

\end{proof}

Formulating this equivalently in terms of leaves in spanning trees,
we obtain:

\begin{corollary}\label{thm:leaves}
As $n\rightarrow\infty$,
$L(Q_n)=2^n (1 - \frac{1}{n} + o(\frac{1}{n})).$
\end{corollary}

\section{Further Study}\label{sec:Concl}

Here are some ideas for continuing research.
We were not able to give a simple enough proof that the
domination number that $\gamma(Q_n)/(2^n/n)\rightarrow1$.
We were hoping to give a self-contained proof of our main
result.  The proof in the literature of this domination
result relies on  rather technical explicit coding constructions.
It would be nice if one could devise an algorithm, or use
probabilistic arguments, to prove the existence of
dominating sets in the hypercube that are as small as
the theorem.

Another question asked by Bezrukov~\cite{Bez} remains
open:  For $n=2^k-1$, starting from the stars given
by the Hamming code, how can one add edges to form a
tree with the most leaves (the smallest connected
dominating set)?  We have seen that for large $k$
the connected dominating set will have
size between 2 and 3 times $2^n/n$.  How can one
add edges efficiently, to get close to the lower bound?

What can one say about a more general class of graphs?
For instance, one could consider domination and
connected domination in a generalized grid (box) graph,
such as a Cartesian product of $n$ paths on $p$ vertices.
This graph on $p^n$ vertices is the hypercube when
$p=2$.  Perhaps the more natural graph to study is
a product of $n$ cycles on $p$ vertices.
Note that for $p=4$ it is the same graph as $Q_{2n}$.
Edenfield~\cite{Ede} recently studied products of
cycles and products of complete graphs, both
generalizations of the hypergraph questions in this paper.

Joshua Cooper suggests considering powers of graphs.
That is, for a graph $G=(V,E)$, such as the hypercube,
fix integer $r>0$ and consider the same questions as
before, but for the graph $G^r$:  This graph also
has vertex set $V$, but now edges join vertices at distance
at most $r$ in $G$.  This is motivated by covering
codes of radius $r$.

\section{Acknowledgements}\label{sec:Concl}
We express gratitude to Sergei Bezrukov for sharing his
questions and notes that led to this study.
Particular thanks go to colleague Joshua Cooper
for bringing to our attention the key domination result
$\gamma(Q_n)\sim 2^n/n$ as presented
in~\cite{CHLL}.

We appreciate support for travel to Taiwan to work on this
project from  Mathematics Research Promotion Center Grant 108-17 and
Ministry of Science and Technology Grant 107-2115-M-005-002-MY2.


\end{document}